\documentclass[11pt]{amsart}
\usepackage{amsmath,amsfonts}
\usepackage{amssymb}
\usepackage{amsthm}
\usepackage[all,cmtip,line]{xy}
\usepackage{array}
\usepackage{overpic}
\usepackage{graphicx}
\usepackage{amsmath,amsthm,amsfonts}
\usepackage{hyperref}

\newtheorem{theorem}{Theorem}
\newtheorem*{main theorem}{Main Theorem}
\newtheorem*{main lemma}{Main Lemma}

\newtheorem{conditional theorem}[theorem]{Conditional Theorem}

\newtheorem{corollary}[theorem]{Corollary}

\newtheorem{conditional corollary}[theorem]{Conditional Corollary}

\newtheorem{lemma}[theorem]{Lemma}

\theoremstyle{definition}
\newtheorem{example}[theorem]{Example}

\newtheorem{definition}[theorem]{Definition}
\newtheorem{definition-lemma}[theorem]{Definition-Lemma}

\newtheorem{definition-proposition}[theorem]{Definition-Proposition}

\theoremstyle{remark}
\newtheorem*{remark}{Remark}

\begin{document}

\title{Strong rational connectedness of Toric Varieties}
\author{Yifei Chen}
\address{Y. Chen: Department of Mathematics, The Johns Hopkins University. Baltimore,
MD 21218, USA}
%\date{Sep 14, 2008}
%\tableofcontents
\email{yichen@math.jhu.edu}

\author{Vyacheslav Shokurov}
\address{V. V. Shokurov: Department of Mathematics, The Johns Hopkins University. Baltimore, MD 21218,
USA} \address{ Steklov Mathematics Institute, Russian Academy of
Sciences, Cubkina Str. 8, 119991, Moscow, Russia}
\email{shokurov@math.jhu.edu}

\maketitle

\begin{abstract} In this paper, we prove that: For any given finitely many distinct
points $P_1,\ldots,P_r$ and a closed subvariety $S$ of codimension
$\geq 2$ in a complete toric variety over a uncountable
(characteristic 0) algebraically closed field, there exists a
rational curve $f:\mathbb{P}^1\rightarrow X$ passing through
$P_1,\ldots,P_r$, disjoint from $S\setminus \{P_1,\ldots,P_r\}$ (see
Main Theorem). As a corollary, we prove that the smooth loci of
complete toric varieties are strongly rationally connected.
\end{abstract}

\tableofcontents

\section{Introduction}

\footnotetext{Both authors were partially supported by NSF grant
DMS-0701465.}

The concept of rationally connected varieties is independently invented by Koll\'{a}r-Miyaoka-Mori
(\cite{kmm92b}) and Campana (\cite{ca92}). This kind of variety has
interesting arithmetic and geometric properties.

A class of proper rationally connected varieties comes from the
smooth Fano varieties (\cite{ca92}, \cite{kmm92a} or \cite{kol96}).
Shokurov (\cite{sh00}), Zhang (\cite{zh06}),  Hacon and McKernan
(\cite{hm07}) proved that FT (Fano type) varieties are rationally connected.

An interesting  question is whether the smooth locus of a rationally
connected variety is rationally connected. In general the answer of
the question is NO. However, for the FT (or log del Pezzo) surface case, Keel
and McKernan gave an affirmative answer, that is, if $(S,\Delta)$ is
a log del Pezzo surface, then its smooth locus $S^{sm}$ is
rationally connected (\cite{km99}), but this does not imply the
strong rational connectedness.

The concept of strongly rationally connected varieties (see
Definition \ref{D:SRC}) was first introduced by Hassett and
Tschinkel (\cite{ht08}). A proper and smooth separably rationally
connected variety $X$ over an algebraically closed field is strongly
rationally connected (\cite{kmm92b} 2.1, or \cite{kol96} IV.3.9). Xu
(\cite{xu08}) announced that the smooth loci of log del
Pezzo surfaces are not only rationally connected but also strongly
rationally connected, which confirms a conjecture of Hassett and
Tschinkel (\cite{ht08}, Conjecture 20). It is expected that the
smooth locus of an FT variety is strongly rationally connected (cf.
Example \ref{E:ToricFT} and Main Theorem).

Throughout the paper, we are working over an uncountably
algebraically closed field of characteristic 0. It is interesting
that whether Main Theorem holds for any algebraically closed (or perfect) field.

\begin{main theorem}   %\label{T:main}
Let $X$ be a complete toric variety. Let $P_1,\ldots,P_r$ be
finitely many distinct points in $X$ ($P_i$ possibly singular). Then
there is a geometrically free rational curve
$f:\mathbb{P}^1\rightarrow X$ over $P_i,1\leq i\leq r$ (see
Definition \ref{D:GeoFree}). Moreover, $f$ is free over $P_i$ if all
points $P_i$
 are smooth.
\end{main theorem}

Main Theorem can be rephrased as follows:

Let $X$ be a complete toric variety. For any given distinct points
$P_1,\ldots,P_r\in X$ (possibly singular) and any given codimension
$\geq 2$ subvariety $S\subseteq X$, there is a rational curve
$f:\mathbb{P}^1\rightarrow X$ passing through $P_1,\ldots,P_r$,
disjoint from $S\setminus\{P_1,\ldots,P_r\}$.

If all points $P_i$ are smooth, then we get the following corollary.

\begin{corollary} The smooth locus of a complete toric variety is
strongly rationally connected.
\end{corollary}

\section{Preliminaries}

When we say that $x$ is a point of a variety $X$, we  mean that $x$
is a closed point in $X$.

 A \emph{rational curve} is a nonconstant morphism
$f:\mathbb{P}^1\rightarrow X$.

\label{D:FT}  A normal projective variety $X$ is called \emph{FT}
(\emph{Fano Type}) if there exists an effective $\mathbb{Q}$-divisor
$D$, such that $(X,D)$ is klt and $-(K_X+D)$ is ample. See
\cite{psh09} Lemma-Definition 2.6 for  other equivalent
definitions.

Let $N\cong \mathbb{Z}^n$ be a lattice of rank $n$. A \emph{toric
variety} $X(\Delta)$ is associated to a fan $\Delta$, a finite
collection of convex cones $\sigma\subset
N_{\mathbb{R}}:=N\otimes_{\mathbb{Z}}\mathbb{R}$ (see \cite{fu93} or
\cite{od88}).

\begin{example} \label{E:ToricFT} Projective toric varieties are FT.
Let $K$ be the canonical divisor of the projective toric variety
$X(\Delta)$, $T$ be the torus of $X$, and $\Sigma=X\setminus T=\sum
D_i$ be the complement of $T$ in $X$. Then  $K$ is linearly
equivalent to $-\Sigma$. Since $X$ is projective, there is an ample
invariant divisor $L$. Suppose that $L=\sum d_iD_i$. Let the
polytope $\Box_L=\{m\in M|\langle m,e_i\rangle+d_i\geq 0,\forall
e_i\in \Delta(1)\}$, where $M$ is the dual lattice of $N$, and
$\Delta(1)$ is the set consisting of 1-dimensional cones in
$\Delta$. Let $u$ be an element in the interior of $\Box_L$. Let $\chi^u$
 be the corresponding rational
function of $u\in M$ (see \cite{fu93} section 1.3), and div $\chi^u$
be the divisor of $\chi^u$. Then
$D=$div $\chi^u+L$ is effective and ample and has support
$\Sigma$. That is, $D=\sum d'_iD_i$ and all $d_i'>0$.

Let $\epsilon$ be a positive rational number, such that all
coefficients of prime divisors in $\epsilon D$ are strictly less
than $1$. Then $\Sigma-\epsilon D$ is effective. It is easy to check
that $(X,\Sigma-\epsilon D)$ is klt, and $-(K+\Sigma-\epsilon D)\sim
\epsilon D$ is ample. Hence $X$ is FT.
\end{example}

\begin{definition} An \emph{isogeny} of toric varieties is a finite surjective
toric morphism. Toric varieties $X$ and $Y$ are said to be
\emph{isogenous} if there exists an isogeny $X\rightarrow Y$. The
\emph{isogeny class} of a toric variety $X$ is a set consisting of
all toric varieties $Y$ such that $X$ and $Y$ are isogenous.
\end{definition}

\begin{theorem} \label{T:isogeny}Let $f:X\rightarrow Y$ be a finite surjective toric morphism.
Then there exists a finite surjective toric morphism $g:Y\rightarrow
X$.
\end{theorem}
\begin{proof} Let $f:X\rightarrow Y$ be a finite surjective toric
morphism of toric varieties and $\varphi:(N',\Delta')\rightarrow
(N,\Delta)$ be the corresponding map of lattices and fans. Then we
can identify $N'$ as a sublattice of $N$ and $\Delta'=\Delta$.

There is an positive integer $r$ such that $rN$ is a sublattice of
$N'$. Let $g$ be the corresponding toric morphism of
$(rN,\Delta)\rightarrow (N',\Delta)$. Since $(rN,\Delta)$ and
$(N,\Delta)$ induce an isomorphic toric variety, we get
$g:Y\rightarrow X$ is a finite surjective toric morphism.
\end{proof}

The properties of isogeny:

1) Isogeny is an equivalence relation.

2) If a toric variety $Y$ is in the isogeny class of $X$ and
$\mu:X\rightarrow Y$ is the isogeny, then there is a one-to-one
correspondence between the set of orbits $\{O_i^X\}$ of $X$ and the
set of orbits $\{O_i^Y=\mu(O_i^X)\}$ of $Y$. Hence $\dim O_i^X=\dim
O_i^Y$ for all $i$, and the number of orbits is independent of the
choice of toric varieties in an isogeny class of $X$.

\smallskip

A variety $X$ over a characteristic 0 field is rationally connected,
if any two general points $x_1,x_2\in X$ can be connected by a
rational curve of $X$ of a bounded family.

\smallskip

\begin{definition} \label{D:SRC}(\cite{ht08} Definition 14.) A smooth rationally
connected variety $Y$ is \emph{strongly rationally connected} if any
of the following conditions hold:

(1) for each point $y\in Y$, there exists a rational curve
$f:\mathbb{P}^1\rightarrow Y$ joining $y$ and a generic point in
$Y$;

(2) for each point $y\in Y$, there exists a free rational curve
containing $y$;

(3) for any finite collection of points $y_1,\ldots,y_m\in Y$, there
exists a very free rational curve containing the $y_j$ as smooth
points;

(4) for any finite collection of jets $$\text{Spec
}k[\epsilon]/\langle\epsilon^{N+1}\rangle\subset Y,\ \ \
i=1,\ldots,m$$ supported at distinct points $y_1,\ldots,y_m$, there
exists a very free rational curve smooth at $y_1,\ldots,y_m$ and
containing the prescribed jets.
\end{definition}

\begin{definition} \label{D:WeakFree} Let $X$ be a complete normal variety, $B$ be a set of finitely many
closed points in $\mathbb{P}^1$, and $g:B\rightarrow X$ be a
morphism. A rational curve $f:\mathbb{P}^1\rightarrow X$ is called
\emph{weakly free} over $g$ if there exist an irreducible family of
rational curves $T$ and an evaluation morphism ev:
$\mathbb{P}^1\times T\rightarrow X$ such that

1) $f=f_{t_0}=$ ev$|_{\mathbb{P}^1\times t_0}$ for some $t_0\in T$,

2) for any $t\in T$, $f_t=$ ev$|_{\mathbb{P}^1\times t}$ is a
rational curve and $f_t|_B=g$,

3) the evaluation morphism ev: $\mathbb{P}^1\times T\rightarrow X$
by ev$(x,t)=f_t(x)$ is dominant.

We say that a \emph{rational curve} $f':\mathbb{P}^1\rightarrow X$
is a \emph{general deformation} of $f$, or $f'$ is a
\emph{sufficiently general weakly free rational curve}, if there is
an open dense subset $U$ of $T$, such that $f'=f_t$ and $t\in
U\subseteq T$. We say that a \emph{weakly free rational curve}
$g:\mathbb{P}^1\rightarrow X$ is a \emph{general deformation of}
$f$, if there is an irreducible family $T'$, such that $T\cap T'$
contains an open dense subset in $T$,  $g=g_{t'}$ for some $t'\in
T'$ and $g$ is weakly free in its own family.
\end{definition}

\begin{definition} \label{D:GeoFree} Let $X$ be a complete normal variety, $B$ be a set of finitely many
closed points in $\mathbb{P}^1$, and $g:B\rightarrow X$  be a
morphism.  A rational curve $f:\mathbb{P}^1\rightarrow X$ is called
\emph{geometrically free} over $g$ if there exist an irreducible
family of rational curves $T$ and an evaluation morphism ev:
$\mathbb{P}^1\times T\rightarrow X$ such that

1) $f=f_{t_0}=$ ev$|_{\mathbb{P}^1\times{t_0}}$  for some $t_0\in
T$,

2) for any $t\in T$, $f_t=$ ev$|_{\mathbb{P}^1\times t}$ is a
rational curve and $f_t|_B=g$,

3) for any codimension 2 subvariety $Z$ in $X$,
$f_t(\mathbb{P}^1)\cap Z\subseteq g(B)$ for general $t\in T$
(general meaning $t$ belongs to a dense open subset in $T$,
depending on $Z$).

\smallskip

If $X$ is smooth over an uncountable field of characteristic 0, then
weak freeness over $g$ is equivalent to usual freeness over $g$ if
$|B|\leq 2$.
\end{definition}

\begin{remark}
In our application, we usually assume $g$ is one-to-one. Let
$P_i=g(Q_i)$ where $B=\{Q_i\}$.  Without confusion, we say $f$ is
geometrically free over $\{P_i\}$ (resp. weakly free over $\{P_i\}$)
instead of saying that $f$ is geometrically free over $g$ (resp.
weakly free over $g$).
\end{remark}

Weak freeness and geometric freeness are  generalizations of usual
freeness (see \cite{kol96} II.3.1 Definition) if the curve passes
through singularities. To consider weakly free rational curves or
geometrically free rational curves, we think of them as general
members in a certain family. In particular, we can suppose that the
morphism ev is flat.

\begin{example} Let $X$ be a projective cone over a conic. Let $T$ be the family of
all lines through the vertex $O$. Then $l\in T$ is not free. However
$l$ is weakly free and geometrically free over $O$ by construction.
\end{example}

\smallskip

We need a resolution as follows.

\begin{theorem}\label{L:specialresolution} Let $X$ be a toric
variety. Let $\Sigma$ be the invariant locus of $X$. Let
$P_1,\ldots,P_r\in X$ be $r$ points. Let $f:\mathbb{P}^1\rightarrow
X$ be a sufficiently general weakly free rational curve over
$P_1,\ldots,P_r$. Then there exists a resolution
$\pi:\tilde{X}\rightarrow X$, such that

1) $\pi^{-1}(\Sigma\cup\{P_i\})$ is a divisor with simple normal
crossing;

2) $\pi^{-1}(P_j)\subseteq \pi^{-1}(\Sigma\cup\{P_i\})$ is a divisor
for each point $P_j$;

3) $\pi:\tilde{X}\rightarrow X$ is an isomorphism over
$X\setminus(\emph{Sing }X\cup \{P_i\})$;

4) sufficiently general  $\tilde{f}(\mathbb{P}^1)$ intersects
$\pi^{-1}(\Sigma\cup\{P_i\})$ over each $P_j$ only in  divisorial
points of $\pi^{-1}(\Sigma\cup\{P_i\})$, where
$\tilde{f}:\mathbb{P}^1\rightarrow\tilde{X}$ is the proper birational transformation
of a general deformation of $f$ and is a (weakly)
free rational curve.

More generally, let $f_j:\mathbb{P}^1\rightarrow X$, $1\leq j\leq m$
be finitely many sufficiently general weakly free rational curve
over a subset of $\{P_i\}$, where $\{P_i\}$ is a set of finitely many
distinct points in $X$. Then there exists a resolution
$\pi:\tilde{X}\rightarrow X$ such that

1') $\pi^{-1}(\Sigma\cup \{P_i\})$ is a divisor with simple normal
crossing;

2') $\pi^{-1}(P_i)\subseteq \pi^{-1}(\Sigma\cup\{P_i\})$ is a
divisor for each point $P_i$;

3') $\pi:\tilde{X}\rightarrow X$ is an isomorphism over
$X\setminus(\emph{Sing }X\cup \{P_i\})$;

4') For each $j$, sufficiently general  $\tilde{f}_j(\mathbb{P}^1)$ intersects
$\pi^{-1}(\Sigma\cup\{P_i\})$ over each $P_i$ only in  divisorial
points of $\pi^{-1}(\Sigma\cup\{P_i\})$, where
$\tilde{f}_j:\mathbb{P}^1\rightarrow\tilde{X}$ is the proper birational transformation
 of a general deformation of $f_j$ and
is a (weakly) free rational curve.
\end{theorem}

\begin{proof} When the
ground field is of characteristic 0, 1)-3) follow from usual facts
in the resolution theory, e.g. see \cite{km98} Theorem 0.2. However,
in the toric or toroidal case, the same result holds for any field.
More precisely, if all $P_i$ are invariant, we can use a toric
resolution. If some $P_i$ are not invariant, they can be converted
into toroidal  invariant points $P_i$ after a toroidalization.

\smallskip

We say that $\tilde{f}(\mathbb{P}^1)$ intersects
$\pi^{-1}(\Sigma\cup\{P_i\})$ over each $P_i$ in a  divisorial point
$x$ if  $x$ belongs to only one prime divisor of
$\pi^{-1}(\Sigma\cup \{P_i\})$ for some $i$ and the prime divisor is over $P_i$. To
fulfill 4), we need extra resolution over intersections of the
divisorial components of $\pi^{-1}(\Sigma\cup\{P_i\})$ through which
general $\tilde{f}$ is passing over $P_i$. Termination of such
resolution follows from an estimation by the multiplicities of
intersection for $f(\mathbb{P}^1)$ with $\Sigma$. The last resolution is independent
of the choice of a general rational curve by Lemma \ref{L:sufficentlygeneral} below. However
it depends on the choice of intersections of divisorial components. For more details,
see the proof of Lemma 4.3.4 in \cite{ch09}.

For the general statement, we can get 1')-3') in a similar manner above. To
fulfill 4'), we just need extra resolutions over each point $P_i$.\end{proof}

We discuss some examples of rational curves on projective spaces and
quotient projective spaces.

\begin{example} \label{E:RCProj} For any given subvariety $S$ of codimension $\geq 2$
in $\mathbb{P}^n$, any points $P_1,\ldots,P_r\in \mathbb{P}^n$, and
any integer $d\geq r$, there exists a rational curve $C$ of degree
$d$, such that each $P_i\in C$ and $C\cap S=\emptyset$.

Indeed, we can construct a tree $T$ with $r$ branches, such that
each $P_i$ is a smooth points on a unique branch and disjoint from
$S$. The tree can be smoothed into a rational curve $C$ passing
through $P_1,\ldots,P_r$, disjoint from $S$. The rational curve $C$
has degree $r$. For $d\geq r$, we can attach $d-r$ rational curves
to the tree $T$, and smooth it.
\end{example}

Applying Example \ref{E:RCProj}, we get

\begin{example} Let $\pi:\mathbb{P}^n\rightarrow X$ be a finite
morphism,  $S$ be a codimension $\geq 2$ subvariety in $X$, and
$\{P_i\}_{i=1}^m$ be a set of $m$ points outside $S$. Then there
exists a rational curve $C$, such that each $P_i\in C$ and $C\cap
S=\emptyset$.
\end{example}

In particular, the same result holds if  $X$ is a quotient space
$\mathbb{P}^n/G$, where $G$ is a finite group, for example, if $X$
is a weighted projective space. It is well known that if $X$ is a
complete $\mathbb{Q}$-factorial toric variety with Picard number
one, then there exist a weighted projective space $Y$ and a finite
toric morphism $\pi:Y\rightarrow X$. So the same result holds for
rational curves on a complete $\mathbb{Q}$-factorial toric variety
with Picard number one. It is a very special case of our Main
Theorem.

\section{Proof of Main Theorem}

In this section we prove Main Theorem. Let us first prove Main
Lemma, which is a special weak case of Main Theorem.

\begin{main lemma}
Let $X$ be a complete toric variety. Let $P,Q\in X$ be two distinct
points ($P,Q$ possibly singular). Let $S\subseteq X$ be a closed
subvariety of codimension $\geq 2$. Then there exists a weakly free
rational curve on $X$ over $P,Q$, disjoint from $S\setminus\{P,Q\}$.
\end{main lemma}

To prove Main Lemma, we need some preliminaries.

\begin{lemma}\label{L:sufficentlygeneral} Let $f$ be a weakly free rational curve on
$X$, and $F_1,\ldots,F_s\subseteq X$ be $s$ proper irreducible
subvarieties in $X$. Then there exist $s',0\leq s'\leq s$, subvarieties among $\{F_j\}$
(after renumbering we assume they are $F_1,\ldots,F_{s'}$) such that
 a general deformation of $f$ intersects $F_1,\ldots,F_{s'}$, and is
disjoint from $F_{s'+1},\ldots, F_s$.
\end{lemma}

The proof of this Lemma is a standard exercise in incidence relations. See \cite{ch09} Lemma 4.3.2 for a detailed proof.

\begin{lemma} \label{L:movefromsmvar} Let $X$ be a complete toric variety. Let
$P,Q\in X$ be two points (possibly singular), and $S$ be a closed
subvariety of codimension $\geq 2$. Let $F_1,\ldots,F_{s}$ be all the
irreducible components of \emph{Sing} $X$.  Let
$f:\mathbb{P}^1\rightarrow X$ be a sufficiently general  weakly free
rational curve over $P,Q$. Suppose $f(\mathbb{P}^1)$ intersects
$F_1\setminus\{P,Q\},\ldots,F_{s'}\setminus\{P,Q\}$, and is disjoint
from $F_{s'+1}\setminus\{P,Q\},\ldots,F_s\setminus\{P,Q\}$. Then
there exists a weakly free rational curve $f'$ over $\{P,Q\}$, which is
a general deformation of $f$, such
that $f'(\mathbb{P}^1)$ is disjoint from
$((S\setminus\text{\emph{Sing} }X) \cup F_{s'+1}\cup\ldots\cup
F_{s})\setminus\{P,Q\}$. Moreover, for any fixed closed  subvariety
$Z$ of $X$, if $f(\mathbb{P}^1)\cap (Z\setminus\{P,Q\})=\emptyset$,
then $f'(\mathbb{P}^1)\cap (Z\setminus\{P,Q\})=\emptyset$.
\end{lemma}

\begin{proof} Applying Theorem \ref{L:specialresolution}
to the toric variety $X$
and two points $\{P,Q\}$, we get a resolution $\pi:\tilde{X}\rightarrow
X$ satisfying 1)-3) in the theorem
and a weakly free rational curve $\tilde{f}:\mathbb{P}^1\rightarrow \tilde{X}$
satisfying 4) in the theorem. A general deformation $\tilde{f}'$ of $\tilde{f}$ is weakly free, so
$\tilde{f}'$ is free by \cite{kol96} II.3.11 (Here we need the assumption that the ground
field is uncountable and of characteristic 0.) Moreover,
we can assume that $\tilde{f}'$ is disjoint from $(S\setminus\text{Sing
}X)\setminus \pi^{-1}\{P,Q\}$ by \cite{kol96} II.3.7.

On the other hand, let $\Sigma$ be the invariant locus of $X$. Notice that
Sing $X\subseteq \Sigma$. Then by Theorem \ref{L:specialresolution},
$\tilde{f}(\mathbb{P}^1)$ intersects $\pi^{-1}(\Sigma \cup\{P,Q\})$ divisorially
over $P,Q$, and $\tilde{f}(\mathbb{P}^1)$ is
disjoint from the closure of
$\pi^{-1}(F_{s'+1}\setminus\{P,Q\}),\ldots,\pi^{-1}(F_s\setminus\{P,Q\})$.
So the general deformation $\tilde{f}'$ of $\tilde{f}$ intersects open subsets
of divisors $\pi^{-1}(P)$ and $\pi^{-1}(Q)$, disjoint from the closure
of $((S\setminus \text{Sing }X)\setminus\pi^{-1}\{P,Q\})\cup
\pi^{-1}(F_{s'+1}\setminus\{P,Q\})\cup\cdots\cup
\pi^{-1}(F_s\setminus\{P,Q\})$.
We apply Lemma \ref{L:wfreegowfree} by replacing $f'$ by $\tilde{f}'$,
dominant morphism $\mu$ by  $\pi:\tilde{X}\rightarrow X$,
 $\{P_i\}$ by $\{P,Q\}$, and $S$ by $(S\setminus$ Sing $X)\cup F_{s'+1}\cup\cdots\cup F_s$.
Then we get the weakly free rational
curve $f'=\pi
\tilde{f'}:\mathbb{P}^1\rightarrow X$ is a general deformation of $f$
(see Definition \ref{D:WeakFree}), passing through points $P,Q$ and
disjoint from $((S\setminus\text{Sing
}X)\cup F_{s'+1}\cup\cdots\cup F_s)\setminus\{P,Q\}$.

Moreover,
we can assume that $f'$ is a weakly free rational curve over $P,Q$,
 by a base change of the family to which $f'$
belongs (For details, see the proof of  Lemma 4.3.1 in \cite{ch09}).

The last statement can be proved similarly.
\end{proof}

\begin{lemma} \label{L:wfreegowfree} Let $X,X'$ be two complete varieties
with $\dim X>0$. Let $\mu:X'\rightarrow X$ be a dominant morphism.
 Then the image of a weakly free rational curve on $X'$
is weakly free on $X$ in the following sense:

Let $P_1,P_2,\ldots,P_r\in \mu(X)$ be $r$ distinct points, and
$S\subseteq X$ be a closed subvariety. Let $S'=\mu^{-1}S$, and
$P_1',P_2',\ldots,P_r'\in X'$ be points such that $\mu(P_i')=P_i$
for $i=1,\ldots,r$. If $f':\mathbb{P}^1\rightarrow X'$ is a weakly
free rational curve over $P_1',P_2',\ldots,P_r'$, disjoint from
$S'\setminus\{P_1',P_2',\ldots,P_r'\}$, then $f=\mu\circ f''$ is a
weakly free rational curve on $X$ over $P_1,P_2,\ldots,P_r$,
disjoint from $S\setminus\{P_1,P_2,\ldots,P_r\}$, where $f''$
is a general deformation of $f'$.
\end{lemma}

\begin{proof} Since $f'$ is weakly free,  ev: $\mathbb{P}^1\times T'\rightarrow X'$
is dominant, where $T'$ is the family associated to $f'$. Since
$\mu:X'\rightarrow X$ is dominant, ev: $\mathbb{P}^1\times
T'\rightarrow X'\rightarrow X$ is dominant. Hence for general deformation
 $f''\in
T'$ of $f'$, $f=\mu\circ f''$ is a weakly free rational curve on $X$.
\end{proof}

\begin{lemma} \label{L:smoothisogeny} Let $X$ be a $\mathbb{Q}$-factorial toric
variety, and $O$ be a singular orbit of $X$. Then there exists an isogeny $\mu:Y\rightarrow
X$, such that $\mu^{-1}(O)$ is smooth.
\end{lemma}

\begin{proof} Let $(N,\Delta)$ be the lattice and fan associated to $X$. Let $N'$ be
the sublattice generated by the primitive elements of the simplicial cone $\sigma$ such that
$O$ is contained in the affine open subset $\sigma$ corresponding to. Let $Y$ be the toric variety
corresponding to $(N',\Delta)$ and $\mu$ be the natural finite dominant morphism corresponding to $(N',\Delta)\rightarrow
(N,\Delta)$. By construction of $\mu$, $\mu^{-1}(O)$ is smooth.
\end{proof}

\begin{proof}[Proof of Main Lemma] \underline{Step 1.}
After $\mathbb{Q}$-factorization $q:X'\rightarrow X$, we can assume
that $X$ is a complete $\mathbb{Q}$-factorial toric variety
(\cite{fj03} Corollary 3.6). Indeed, a weakly free rational curve on
$X'$ gives a weakly free rational curve on $X$ by Lemma
\ref{L:wfreegowfree}.

\smallskip

\underline{Step 2.} A weakly free rational curve can be moved from any
smooth variety of codimension $\geq 2$ in the sense of Lemma
\ref{L:movefromsmvar}. So we can reduce the proof of Main Lemma to
the case $S=I(X)$, where $I(X)$ denotes the union of orbits of $X$
of codimension $\geq 2$. Since $X$ is a toric variety, Sing
$X\subseteq I(X)$.

Indeed, for any subvariety $S\subseteq X$ of codimension $\geq 2$,
suppose there is a sufficiently general weakly free rational curve
$f:\mathbb{P}^1\rightarrow X$ over $P,Q\in X$, disjoint from
$I(X)\setminus\{P,Q\}$.  Apply Lemma \ref{L:movefromsmvar} to the
subvariety $S$, and the weakly free rational curve $f$. Since Sing
$X\subseteq I(X)$, $s'=0$ in Lemma \ref{L:movefromsmvar}, that is,
$f(\mathbb{P}^1)$ is disjoint from
$F_1\setminus\{P,Q\},\ldots,F_s\setminus\{P,Q\}$. Then there exists
a weakly free rational curve $f'$, which is a general deformation of
$f$, such that $f'(\mathbb{P}^1)$ is disjoint from
$((S\setminus\text{Sing }X)\cup F_1\cup\ldots\cup
F_s)\setminus\{P,Q\}=((S\setminus\text{Sing }X)\cup\text{ Sing
}X)\setminus\{P,Q\}=S\setminus\{P,Q\}$.

\smallskip

\underline{Step 3.} Suppose that $I(X)$ consists of $\tilde{s}$ distinct orbits
$O_1,\ldots,O_{\tilde{s}}$. Let $f:\mathbb{P}^1\rightarrow X$ be a sufficiently general
weakly free rational curve over $P,Q$. By Lemma
\ref{L:sufficentlygeneral}, we can assume that $f(\mathbb{P}^1)$
intersects with $O_1\setminus\{P,Q\},\ldots,O_{s'}\setminus\{P,Q\}$,
and is disjoint from
$O_{s'+1}\setminus\{P,Q\},\ldots,O_{\tilde{s}}\setminus\{P,Q\}$ for
some $s'$.

Notice that $s'$ depends on the points $P,Q$ and the variety $X$.
However, since $s'$ is bounded by $\tilde{s}$, and $\tilde{s}$ is
independent of choice of $X$ in an isogeny class, there exists an $\bar{s}$ such that
for any toric variety $Y$ in the isogeny class of $X$, and two
distinct points $P',Q'\in Y$, there exists a weakly free rational curve
$f'_{\bar{s}}:\mathbb{P}^1\rightarrow Y$ over $P',Q'$,
 such that $f'_{\bar{s}}(\mathbb{P}^1)$ intersects with at
most
$O_1^Y\setminus\{P',Q'\},\ldots,O_{\bar{s}}^Y\setminus\{P',Q'\}$,
and is disjoint from
$O_{\bar{s}+1}^Y\setminus\{P,Q\},\ldots,O_{\tilde{s}}^Y\setminus\{P,Q\}$,
where $O_i^Y$ are orbits of $Y$ of codimension $\geq 2$.
Furthermore, we can assume that $\dim O_1^Y\geq \dim O_2^Y\geq
\cdots\geq \dim O_{s'}^Y\geq \dim O_{s'+1}^Y\geq \cdots\geq\dim
O_{\tilde{s}}^Y$. This order is good for us, because  $\cup_{j\geq
 s}O_{j}^Y$ is closed for any $s$.

\smallskip

We fix a complete toric variety $X$, two points $P,Q$ and a weakly
free rational curve $f_{\bar{s}}$ over $P,Q$. By Lemma
\ref{L:wfreegowfree} and  \ref{L:smoothisogeny}, we can suppose that
the orbit $O_{\bar{s}}$ is smooth. Indeed, by Lemma
\ref{L:smoothisogeny}, there is an isogeny $\mu:Y\rightarrow X$ such
that $O_{\bar{s}}^Y=\mu^{-1}(O_{\bar{s}})$ is smooth. Let $P',Q'\in
Y$ such that $\mu(P')=P,\mu(Q')=Q$. Then existence of a weakly free
rational curve $f':\mathbb{P}^1\rightarrow Y$ over $P',Q'$, disjoint
from $O_{\bar{s}}^Y\cup \cdots \cup O_{\tilde{s}}^Y$, implies
existence of a weakly free rational curve $f:\mathbb{P}^1\rightarrow
X$ over $P,Q$, disjoint from $O_{\bar{s}}\cup\cdots\cup
O_{\tilde{s}}$,  by Lemma \ref{L:wfreegowfree} with
$X'=Y,\{P_i\}=\{P,Q\}$ and $S=O_{\bar{s}}^Y\cup
O_{\bar{s}+1}^Y\cup\cdots\cup O_{\tilde{s}}^Y$.

\smallskip

\underline{Step 4.} Now, we prove that there is a weakly free
rational curve $f_{\bar{s}-1}$ over $P,Q$ such that
$f_{\bar{s}-1}(\mathbb{P}^1)$ intersects at most
$O_1\setminus\{P,Q\},\ldots,O_{\bar{s}-1}\setminus\{P,Q\}$, and is
disjoint from
$O_{\bar{s}}\setminus\{P,Q\},\ldots,O_{\tilde{s}}\setminus\{P,Q\}$.
Indeed, we have the following two cases:

1) If $f_{\bar{s}}(\mathbb{P}^1)$ is disjoint from
$O_{\bar{s}}\setminus\{P,Q\}$, then let $f_{\bar{s}-1}=f_{\bar{s}}$.

2) If $f_{\bar{s}}(\mathbb{P}^1)$ intersects
$O_{\bar{s}}\setminus\{P,Q\}$, we apply Lemma \ref{L:movefromsmvar}
with $Z=O_{\bar{s}+1}\cup\cdots\cup O_{\tilde{s}}$ and
$S=O_{\bar{s}}\cup Z$. Notice that $S$ and $Z$ are  closed
subvarieties of $X$ of codimension $\geq 2$, and $O_{\bar{s}}$ is
smooth. In particular, $S\setminus$ Sing $X\supseteq O_{\bar{s}}$.
By assumption, $f_{\bar{s}}(\mathbb{P}^1)\cap
(Z\setminus\{P,Q\})=\emptyset$. Therefore, by the Lemma, there
exists a weakly free rational curve $f_{\bar{s}-1}$ on $X$, which is
a general deformation of $f_{\bar{s}}$, such that
$f_{\bar{s}-1}(\mathbb{P}^1)$ intersects at most
$O_1\setminus\{P,Q\},\ldots,O_{\bar{s}-1}\setminus\{P,Q\}$, and is
disjoint from $(O_{\bar{s}}\cup
Z)\setminus\{P,Q\}=(O_{\bar{s}}\setminus\{P,Q\})\cup(O_{\bar{s}+1}\setminus\{P,Q\})\cup\cdots\cup(O_{\tilde{s}}\setminus\{P,Q\})$.

\smallskip

\underline{Step 5.} By induction on $\bar{s}$, there is a weakly
free rational curve $f_0$ over $P,Q$, disjoint from
$I(X)\setminus\{P,Q\}$.
\end{proof}

\begin{proof}[Proof of Main Theorem]

\underline{Step 1.} First, let us consider $S=$ Sing $X$.

There is a free rational curve $f_0:C_0\cong\mathbb{P}^1\rightarrow
X$ disjoint from $\{P_i\}\cup S$. Indeed, we can apply Main Lemma to
the subvariety $\{P_i\}\cup S$ and any two smooth points
$P,Q\not\in\{P_i\}\cup S$ in $X$.  Since $f_0(\mathbb{P}^1)$ is in
the smooth locus of $X$,  $f_0$ is free and disjoint from
$\{P_i\}\cup S$.

\begin{center}
\begin{overpic}[scale=0.8]{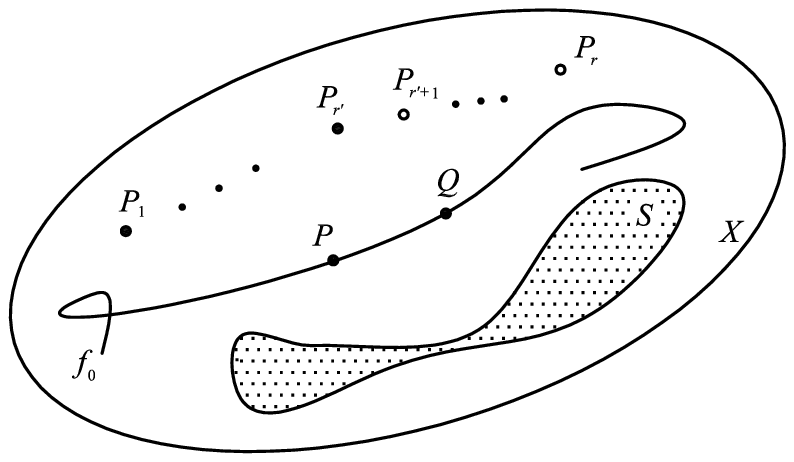}
\end{overpic}
\end{center}

We construct a comb of smooth rational
curves $C$ and a morphism $f:C\rightarrow X'$ as follows.

\textbf{I.} Assume that $P_1,\ldots,P_{r'}$ are smooth points for
some $r'$, $1\leq r'\leq r$, and $P_{r'+1},\ldots,P_r$ are singular
points of $X$. Choosing points $t_1,\ldots,t_r\in C_0$, such that
$P_i'=f_0(t_i)\in X$ are distinct. For each $j$, applying the Main
Lemma  to $S=$
Sing $X\cup\{P_i\}$ and points $P=P_j,Q=P_j'$, there is a weakly free rational
curve $f_j:C_j\cong\mathbb{P}^1 \rightarrow X$ over $P_j,P_j'$ for
each $1\leq j\leq r$, disjoint from $S\setminus\{P_j,P_j'\}$.

\begin{center}
\begin{overpic}[scale=0.8]{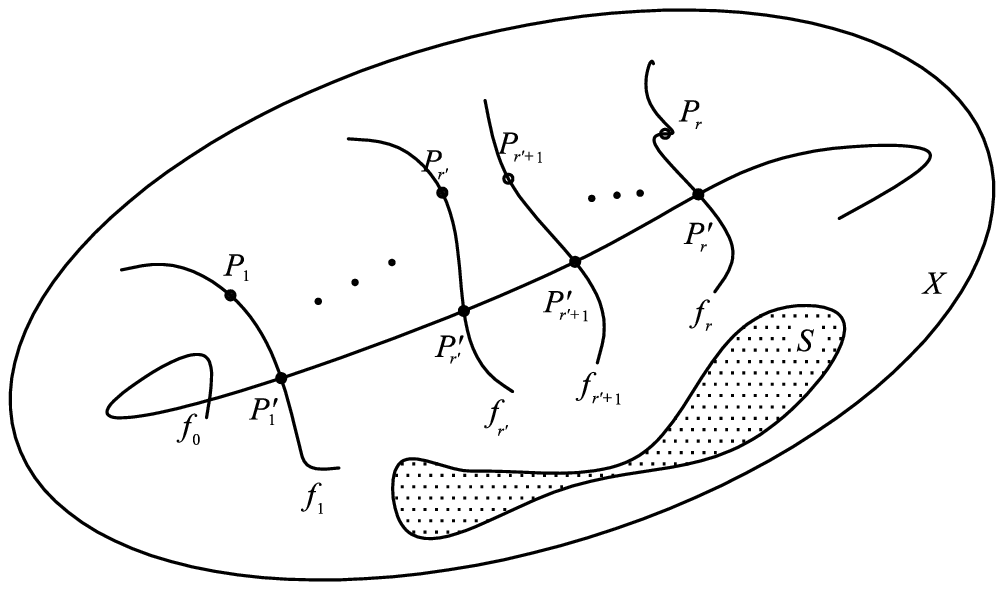}
\end{overpic}
\end{center}

Applying the general statement of Theorem
\ref{L:specialresolution} to weakly free rational curves $f_0,f_1,\ldots,f_r$
and the set $\{P_i\}=\{P_i\}_{i\geq r'+1}$, we get a resolution $\pi:X'\rightarrow X$.

For each $1\leq i\leq r'$, since $P_i$ and $P_i'$ are smooth points,
$f_i(\mathbb{P}^1)$ is contained in the smooth locus of $X$.
Therefore $f_i$ is free for each $1\leq i\leq r'$ by \cite{kol96} II.3.11.
 We identify the curve $f_i:C_i\cong \mathbb{P}^1\rightarrow X$
birationally with a free rational curve
$f_i:C_i\cong\mathbb{P}^1\rightarrow X'$. We also identify $P_i\in
X$ with $P_i\in X'$ for $1\leq i\leq r'$, and $P_i'\in X$ with $P_i'\in
X'$ for $1\leq i\leq r$. More precisely, $f_i(0_i)=P_i$, where $0_i\in C_i,1\leq i\leq r'$,
and $f_i(\infty_i)=P_i'$ where $\infty_i\in C_i,1\leq i\leq r$.

For each $r'+1\leq j\leq r$, $P_j$ is singular. Let
$f_j':C_j\cong\mathbb{P}^1\rightarrow X'$ be the proper birational transformation
 of a sufficiently general deformation of $f_j$. Since
$\pi:X'\rightarrow X$ is a resolution in Theorem
\ref{L:specialresolution}, $f_j'(C_j)$ intersects $\pi^{-1}P_j$
divisorially over $P_j$ for $r'+1\leq j\leq r$, and is disjoint from
the closure of $\pi^{-1}(S\setminus\{P_i\})$. Let $Q_j$ be a point
in $f_j'(C_j)\cap \pi^{-1}P_j$ over $P_j$ for $r'+1\leq j\leq r$. We
can suppose that $f_i$ is very free for $1\leq i\leq r'$ and $f_j'$
is very free for $r'+1\leq j\leq r$ by \cite{kmm92a} 1.1. or
\cite{kol96} II.3.11.

By construction of $f_i,1\leq i\leq r'$ and $f_j',r'+1\leq j\leq r$,
$f_i(C_i)$ and $f_j'(C_j)$ are disjoint from the closure of
$\pi^{-1}(S\setminus\{P_1,\ldots,P_r\})=\pi^{-1}(S\setminus\{P_{r'+1},\ldots,P_r\})$.

\begin{center}
\begin{overpic}[scale=0.8]{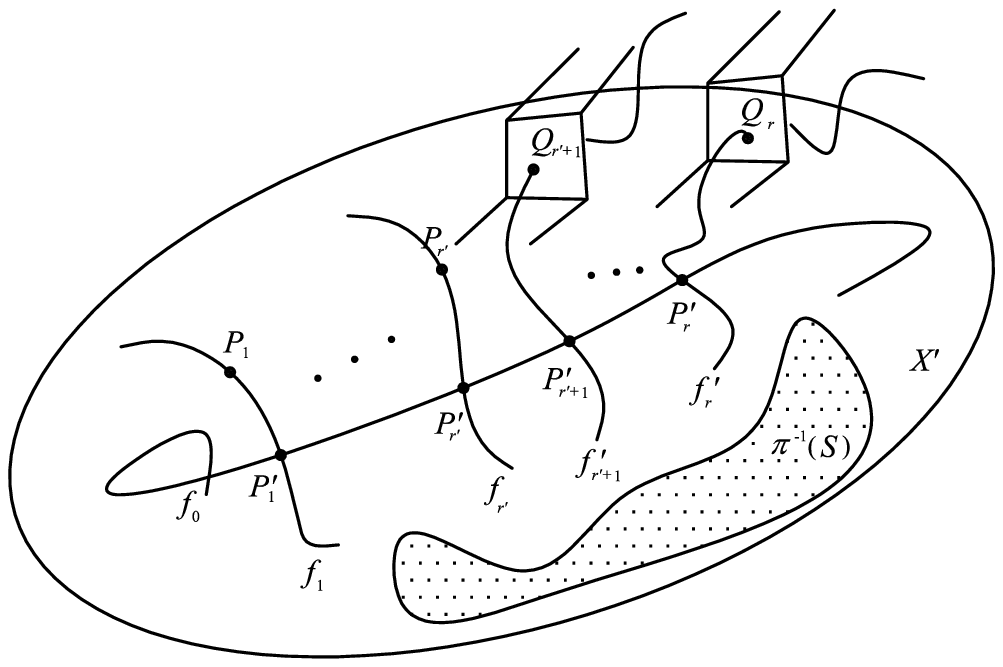}
\end{overpic}
\end{center}

\textbf{II.} Gluing $\cup_{i=0}^r C_i$, we get a comb of smooth
rational curves $C=\sum_{i=0}^r C_i$ and a morphism $f:C\rightarrow
X'$. Indeed, we identify
points $\infty_i\in C_i$ with $t_i\in C_0$ for each $1\leq i\leq r$. Then we have a comb of
smooth rational curves $C=\sum_{i=0}^rC_i$ and a morphism
$f:C\rightarrow X'$ because $f_0(t_i)=f_i(\infty_i)=P_i'$. Notice that $f(C)$
is disjoint from the closure of $\pi^{-1}(S\setminus\{P_1,\ldots,P_r\})$.

In the end, $f:C\rightarrow X'$ can be smoothed into a rational curve
$f':\mathbb{P}^1\rightarrow X'$ such that $f'$ is free over
$P_i,1\leq i\leq r'$ and $Q_j,r'+1\leq j\leq r$, and is disjoint
from the closure of $\pi^{-1}(S\setminus\{P_1,\ldots,P_r\})$ (We can generalize the
proof of  \cite{kol96} II.7.6 for comb to get $f'$ is a free
rational curve over $\{P_1,\ldots,P_{r'},Q_{r'+1},\ldots,Q_{r}\}$, not only with
$\{P_1,\ldots,P_{r'},Q_{r'+1},\ldots,Q_{r}\}$ fixed,
as stated in \cite{kol96} II.7.6. Or we can attach additional
rational curves to enlarge of the family of $f'$, such that $f'$ is a
free rational curve over $\{P_1,\ldots,P_{r'},Q_{r'+1},\ldots,Q_{r}\}$ after a base change).

\begin{center}
\begin{overpic}[scale=0.8]{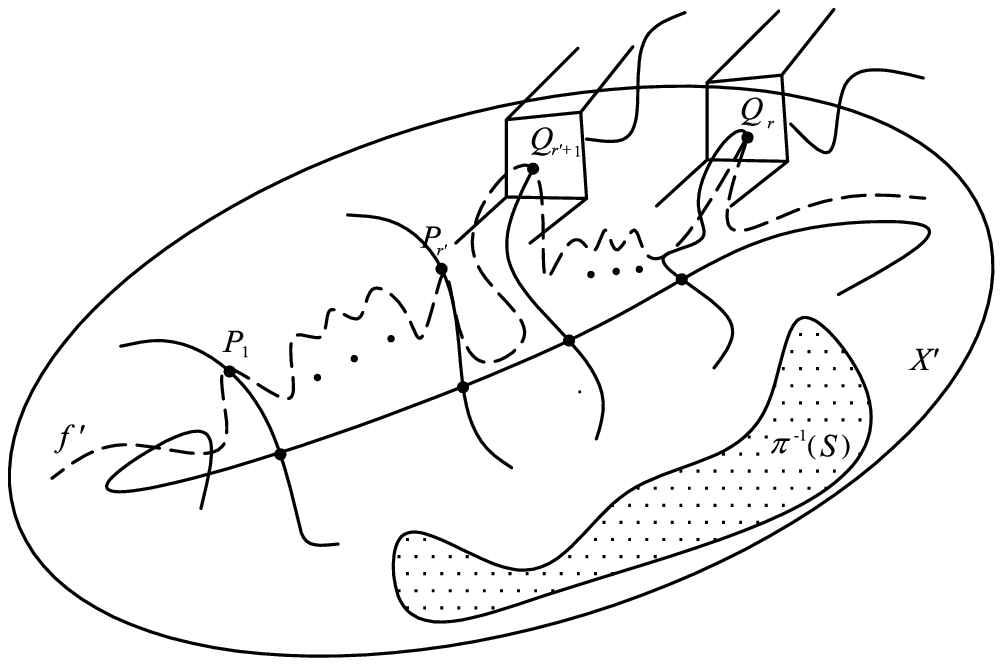}
\end{overpic}
\end{center}

\underline{Step 2.} Now we consider any closed subvariety $S$ of
codimension $\geq 2$.

By Step 1, there is a free rational curve
$f':\mathbb{P}^1\rightarrow X'$ over $P_1,\ldots,P_{r'}$, $Q_{r'+1},\ldots,Q_r$,
disjoint from the closure of
$\pi^{-1}(\text{Sing }X\setminus\{P_1,\ldots,P_r\})$, where $\pi:X'\rightarrow X$ is the resolution
in Step 1. On the other hand, $\pi^{-1}((S\setminus\text{Sing }X)
\setminus\{P_1,\ldots,P_r\})$ is a codimension $\geq 2$ subvariety on $X'$ by Theorem \ref{L:specialresolution} 3').
So a general deformation $f''$ of $f'$ is free over $P_1,\ldots,P_{r'},Q_{r'+1},\ldots,Q_r$,
disjoint from $\pi^{-1}((S\setminus\text{Sing }X)
\setminus\{P_1,\ldots,P_r\})$ by \cite{kol96} II.3.7. Since $f'$ is disjoint from the closure of
$\pi^{-1}(\text{Sing }X\setminus\{P_1,\ldots,P_r\})$, $f''$ is disjoint from $\pi^{-1}(\text{Sing }X\setminus\{P_1,\ldots,P_r\})$.
Hence $f''$ is disjoint from
$\pi^{-1}(\text{Sing }X\setminus\{P_1,\ldots,P_r\})\cup \pi^{-1}((S\setminus\text{Sing }X)
\setminus\{P_1,\ldots,P_r\})=\pi^{-1}(S\setminus\{P_1,\ldots,P_r\})$.
Therefore, $\pi f''$ is a general deformation of
$\pi f'$ over $P_1,\ldots,P_r$, disjoint from $S\setminus\{P_1,\ldots,P_r\}$,
and thus $\pi f'$ is a
geometrically free rational curve over $P_1,\ldots,P_r$ on $X$.
\end{proof}

\end{document}